\documentclass[letter,11pt,dvipsnames]{amsart}
\newcommand{\ignore}[1]{}

\usepackage{verbatim,lipsum} 
\usepackage[x11names, rgb]{xcolor}
\usepackage[utf8]{inputenc}
\usepackage{tikz}
\usetikzlibrary{snakes,arrows,shapes}
\usepackage{fullpage}
\usepackage[colorlinks=false, pagebackref]{hyperref}  
\usepackage{microtype}
\usepackage{textcomp}
\usepackage{amsmath, amsfonts, amssymb, amsthm}
\usepackage{mathtools, mathdots}
\usepackage{mhsetup}
\usepackage{algorithm, algpseudocode, algorithmicx, float}
\usepackage{enumerate}
\usepackage{comment}

\DeclarePairedDelimiter\ceil{\lceil}{\rceil}

\newtheorem{theorem}{Theorem}

\newtheorem{corollary}[theorem]{Corollary}
\newtheorem{lemma}[theorem]{Lemma}
\newtheorem{definition}[theorem]{Definition}
\newtheorem{proposition}{Proposition}

\newtheorem{fact}{Fact}

\newtheorem{claim}[theorem]{Claim}

\newcommand{\claimproof}[2]%
{\noindent{\em Proof of Claim \ref{#1}.}
#2\hspace*{\fill}$\Box$~~~~\vspace{3.5mm} }

\newcommand{\F}{\mathbb{F}}
\newcommand{\N}{\mathbb{N}}

\newcommand{\Q}{\mathbb{Q}}
\newcommand{\Z}{\mathbb{Z}}
\newcommand{\C}{\mathbb{C}}

\newcommand{\rad}{\text{rad}}

\newcommand{\mb}{\mathbf}



\makeatletter
\newcommand{\algmargin}{\the\ALG@thistlm}
\makeatother
\newlength{\whilewidth}
\algdef{SE}[parWHILE]{parWhile}{EndparWhile}[1]
  {\parbox[t]{\dimexpr\linewidth-\algmargin}{%
     \hangindent\whilewidth\strut\algorithmicwhile\ #1\ \algorithmicdo\strut}}{\algorithmicend\ \algorithmicwhile}%
\algnewcommand{\parState}[1]{\State%
  \parbox[t]{\dimexpr\linewidth-\algmargin}{\strut #1\strut}}

\begin{document}

\pagenumbering{gobble}


\title{\bf Computing Igusa's local zeta function of univariates in deterministic polynomial-time}


\author{Ashish Dwivedi} 
\address{Department of Computer Science \& Engineering, Indian Institute of Technology Kanpur} 
\email{ashish@cse.iitk.ac.in} 
%
\author{Nitin Saxena} 
\address{Department of Computer Science \& Engineering, Indian Institute of Technology Kanpur} 
\email{nitin@cse.iitk.ac.in} 

\date{}
\maketitle

\begin{abstract}
Igusa's local zeta function $Z_{f,p}(s)$ is the generating function that counts the number of integral roots, $N_{k}(f)$, of $f(\mathbf x) \bmod p^k$, for all $k$. It is a famous result, in analytic number theory, that $Z_{f,p}$ is a rational function in $\Q(p^s)$. We give an elementary proof of this fact for a univariate polynomial $f$. Our proof is constructive as it gives a closed-form expression for the number of roots $N_{k}(f)$. 

Our proof, when combined with the recent root-counting algorithm of (Dwivedi, Mittal, Saxena, CCC, 2019), yields the first deterministic poly($|f|, \log p$) time algorithm to compute $Z_{f,p}(s)$. Previously, an algorithm was known only in the case when $f$ completely splits over $\Q_p$; it required the rational roots to use the concept of generating function of a tree (Z{\'u}{\~n}iga-Galindo, J.Int.Seq., 2003).
\end{abstract}

\vspace{-.30mm}
\noindent
{\bf MSC 2010 Classification:} Primary-- 11S40, 68Q01, 68W30. Secondary-- 11Y16, 14G50.

\noindent
{\bf Keywords:} Igusa, local, zeta function, discriminant, valuation, deterministic, root, counting, modulo, prime-power.  


\pagenumbering{arabic}

\vspace{-1mm}

\section{Introduction}

Over the years, the study of zeta functions have played a foundational role in the development of mathematics. They also find great applications in diverse science disciplines, in particular-- machine learning \cite{watanabe2009algebraic}, cryptography \cite{anshel1997zeta, anshel1998multi}, quantum cryptography\cite{li2005application}, statistics \cite{watanabe2009algebraic, lin2017ideal}, theoretical physics \cite{hawking1977zeta,reshetikhin2015combinatorial}, string theory \cite{polchinski1998string}, quantum field theory \cite{elizalde1996applications, hawking1977zeta}, biology \cite{robson2005clinical, willison2019intracellular}.
Basically, a zeta function counts some mathematical objects. Often zeta functions show special analytic, or algebraic properties, the study of which reveals various hidden and striking information about the encoded object. 

A classic example is the famous Riemann zeta function \cite{riemann1859ueber} (also known as Euler-Riemann zeta function) which encodes the density, and distribution, of prime numbers \cite{conrey2003riemann, titchmarsh1986theory}. Later many {\em local} (i.e.~associated to a specific prime $p$) zeta functions had been studied. Eg.~Hasse-Weil zeta function \cite{weil1948varietes,weil1949numbers} which encodes the count of zeros of a system of polynomial equations over finite fields (of a specific characteristic $p$). It lead to the development of modern algebraic-geometry (see \cite{deligne1974conjecture,grothendieck1964formule}). 

In this paper we are interested in a different local zeta function known as Igusa's local zeta function. It encodes the count of roots modulo prime powers of a given polynomial defined over a local field. Formally, {\em Igusa's local zeta function} $Z_{f,p}(s)$, attached to a polynomial over $p$-adic integers, $f(\mb{x})\in \Z_p[x_1,\ldots,x_n]$, is defined as, 
$$ Z_{f,p}(s) \;:=\; \int_{\Z_{p}^{n}} \lvert f(\mb{x}) \rvert_{p}^{s} \cdot\lvert d\mathbf x\rvert $$ 
where $s\in \C$ with $Re(s)>0$, $\lvert .\rvert_p$ denotes the absolute value over $p$-adic numbers $\Q_p$, and $\lvert d\mathbf x\rvert$ denotes the Haar measure on $\Q_{p}^{n}$ normalized so that $\Z_{p}^{n}$ has measure $1$.

Weil \cite{weil1964certains, weil1965formule} defined these zeta functions inspired by those of Riemann. Later they were studied extensively by Igusa \cite{igusa1974complex, igusa1975complex, igusa1978lectures}. Using the method of resolution of singularities, Igusa proved that $Z_{f,p}(s)$ converges to a rational function.  Later the convergence was proved by Denef \cite{denef1984rationality} via a different method (namely, $p$-adic cell decomposition). Igusa zeta function is closely related to {\em Poincar\'e series} $P(t)$, attached to $f$ and $p$, defined as \[ P(t) \;:=\; \sum_{i=0}^{\infty} N_{i}(f)\cdot (p^{-n}t)^i \] where $t\in \C$ with $\lvert t\rvert<1$, and $N_i(f)$ is the count on roots of $f\bmod p^i$ (also $N_0(f):=1$). In fact, it has been shown in \cite{igusa2000} that 
$$ P(t) \;=\; \frac{1-t\cdot Z_{f,p}(s)}{1-t}$$ 
with $t=:p^{-s}$. So rationality of $Z_{f,p}(s)$ implies rationality of $P(t)$ and vice versa; thus proving  a conjecture of \cite{borewicz1969si} that $P(t)$ is a rational function. This relation makes the local zeta function interesting in arithmetic geometry (see \cite{igusa2000, denef1991local, meuser2016survey,  cardenal2019introduction} for more on Igusa zeta function).

Many researchers have tried to calculate the expression for Igusa zeta function for various polynomial families  \cite{cowan2017new, robinson1996igusa, veys1997zeta, albis1999elementary, denef2001newton, marko2005igusa, vega2005igusa, ibadula2005plane, saia2005local, zuniga2001igusa, zuniga2003local} and this has lead to the development of various methodologies. For example, stationary phase formula (SPF), Newton polygon method, resolution of singularities, etc. These methods have been fruitful in various other situations \cite{du2000analytic, zuniga2004pseudo, zuniga2006decay, sakellaridis2008unramified, klopsch2009igusa, klopsch2009zeta, zuniga2009local, voll2010functional, segers2011exponential, zuniga2011local}. Though, not much has been said about their algorithmic aspect except in the case of resolution of singularities \cite{bierstone1997canonical, bodnar2000computer, bodnar2000automated, villamayor1989constructiveness}. These algorithms are impractical \cite{bierstone2011effective}. Indeed, the computation of Igusa zeta function for a general multivariate polynomial seems to be an intractable problem since root counting of a multivariate polynomial over a finite field is known to be NP-hard \cite{garey1979computers, ehrenfeucht1990computational}.

In this paper, we focus on the computation of Igusa zeta function when the associated polynomial is {\em univariate}. Igusa zeta function for a univariate polynomial $f$ is connected to root counting of $f$ modulo prime powers $p^k$, which is itself an interesting problem. It has applications in factoring \cite{chistov1987efficient, chistov1994algorithm, cantor2000factoring}, coding theory  \cite{berthomieu2013polynomial, salagean}, elliptic curve cryptography  \cite{lauder2004counting}, arithmetic algebraic-geometry \cite{zuniga2003computing, denef2001newton, denef1991local},  study of root-sets \cite{sierpinski1955remarques, chojnacka1956congruences, bhargava1997p, dearden1997roots, maulik2001root}. After a long series of work \cite{von1996factorization, von1998factoring, klivans1997factoring, salagean, berthomieu2013polynomial, sircana2017factorization, cheng2017counting, kopp2018randomized, dwivedi2019efficiently}, this was recently resolved in \cite{dwivedi2019counting}.

In the case of univariate polynomials one naturally expects an elementary proof of convergence, as well as an efficient algorithm to compute the Igusa zeta function. Our main result is,

\smallskip\noindent
{\em We give the first deterministic polynomial time algorithm to compute the rational function form of Igusa zeta function, associated to a given univariate polynomial $f\in \Z[x]$ and prime $p$.}
 
\smallskip
 To the best of our knowledge,  previously it was achieved for the restricted class of univariate polynomials using methods that were sophisticated, and non-explicit. For example, Z\'u\~niga-Galindo \cite{zuniga2003computing} achieves this, for univariate polynomials which completely split over $\Q$ (with the factorization given in the input), using stationary phase formula (see Section \ref{sec-remarks}). The methods to compute the Igusa zeta function for a multivariate, eg.~Denef \cite{denef1984rationality}, continue to be impractical in the case of univariate polynomials. On the contrary, our approach is elementary, uses explicit methods, and completely solves the problem.

\subsection{Our results}

We will compute the Igusa zeta function $Z_{f,p}(s)$ by finding the related Poincar\'e series $P(t)=: A(t)/B(t)$.

\begin{theorem}
\label{thm-main}
We are given a univariate integral polynomial $f(x)\in \Z[x]$ of degree $d$, with coefficients magnitude bounded by $C\in \N$, and a prime $p$. Then, we compute Poincar\'e series $P(t)=A(t)/B(t)$, associated with $f$ and $p$, in deterministic poly($d,\log C+\log p$) time. 

The degree of the integral polynomial $A(t)$ is $\tilde{O}(d^2\log C)$ and that of $B(t)$ is $O(d)$.
\end{theorem}

\noindent{\bf Remarks-- (1)}
Our method gives an elementary proof of rationality of $Z_{f,p}(s)$ as a function of $t=p^{-s}$. 

{\bf (2)}
Previously, Z\'u\~niga-Galindo \cite{zuniga2003computing} gave a deterministic polynomial time algorithm to compute $Z_{f,p}(s)$, if $f$ completely splits over $\Q$ and the roots are provided. While, our Theorem \ref{thm-main} works for {\em any} input $f\in \Z[x]$ (see Section \ref{sec-remarks} for further discussion).

{\bf (3)}
Cheng et al.~\cite{cheng2017counting} could compute $Z_{f,p}(s)$ in deterministic poly-time, in the special case where the degree of $A(t), B(t)$ is constant.

{\bf (4)}
Dwivedi et al.~\cite{dwivedi2019counting}, using \cite{zuniga2003computing}, remarked that $Z_{f,p}(s)$ could be computed in deterministic poly-time, in the special case when $f$ completely splits over $\Q_p$ {\em without} the roots being provided in the input.
The detailed proof of this claim was not given and the convergence relied on the old method of \cite{zuniga2003computing}.

\smallskip
We achieve the rational form of $Z_{f,p}(s)$ by getting an explicit formula for the number of zeros $N_{k}(f)$, of $f\bmod p^k$, which sheds new light on the properties of the function $N_{k}(\cdot)$. Eventually, it gives an elementary proof of the rationality of the Poincar\'e series $\sum_{i=0}^{\infty} N_{i}(f) \cdot(p^{-1}t)^i$.

\begin{corollary}
\label{cor1}
Let $k$ be large enough, namely, $k\geq k_0:= O(d^2(\log C+\log d))$. Then, we give a closed form expression for $N_k(f)$ (in Theorem \ref{thm-formula}).

Interestingly, if $f$ has non-zero discriminant, then $N_k(f)$ is constant (independent of $k$) for all $k\geq k_0$. 
\end{corollary}

\subsection{Further remarks \& comparison}
\label{sec-remarks}

To the best of our knowledge, there have been very few results on the complexity of computing Igusa zeta function for univariate polynomials \cite{zuniga2003computing, cheng2017counting}. Other very specialized algorithms are for bivariate polynomials (eg.~hyperelliptic curves) \cite{cardenalalgorithm}, and for the polynomial $x^q-a$ \cite{vega2005igusa}. In a recent related work \cite[Appendix~A]{zhu2020trees}, a different proof of rationality of Igusa zeta function for univariate polynomials based on tree based algorithm of \cite{kopp2018randomized} is given.

An old proof technique called {\em stationary phase formula} is the standard method used in literature to compute Igusa zeta function of various families of polynomials. Our work, on the other hand, uses elementary techniques and a tree based root-counting algorithm \cite{dwivedi2019counting}, to compute some fixed parameters (independent of $k$) involved in our formula of $N_k(f)$, for all $k\geq k_0$. 

It is to be noted that just  efficiently computing $N_k(f)$, for `several' $k$, is not enough to compute the rational form of  $Z_{f,p}(s)$; neither does it imply the rationality of $Z_{f,p}(s)$ directly.

Our algorithm is {\em deterministic} and works for general $f\in \Z_p[x]$. For earlier methods to work for $f\in \Z_p[x]$ they may need factoring over $p$-adics $\Z_p$ or $\Q_p$ (for eg. \cite{zuniga2003computing}); but deterministic algorithms there are unknown. See \cite{chistov1987efficient, chistov1994algorithm, cantor2000factoring} for randomized factoring algorithms.

\subsection{Proof idea}

We will compute the rational form of Igusa zeta function via computing the rational form of corresponding Poincar\'e series $P(t):=$ $\sum_{i=0}^{\infty} N_i(f)\cdot (p^{-1} t)^i$. In addition, our method proves that the Poincar\'e series is a rational function of $t$, in case of univariate polynomial $f(x)$, via first principles; instead of using advanced tools like `stationary phase method' or `Newton polygon method' or `resolution of singularity'.

To compute rational form of Poincar\'e series the idea is to compute the coefficient sequence $\{ N_0(f),\ldots,N_k(f),\ldots\}$ in a closed form. That is to say, get an explicit formula for $N_k(f)$, the number of roots of $f\bmod p^k$, only in terms of $k$; with the hope that this will help in getting a rational function for the Poincar\'e series $P(t)$.

Indeed in Theorem \ref{thm-formula}, we show that such a formula exists for each $N_k(f)$ for sufficiently large $k$. We achieve this by establishing a connection among roots of $f\bmod p^k$ and $\Z_p$-roots of $f\in \Z_p[x]$. Let $f$ has $n$ distinct $\Z_p$-roots $\alpha_1,\ldots,\alpha_n$. An important concept we define is that of `neighborhood' of an $\alpha_i$ mod $p^k$ (Definition \ref{def-S_ki}); these are basically roots of $f\bmod p^k$ `associated' to $\alpha_i$.
 In Lemma \ref{lemma-uni-assoc}, we show that {\em each} root $\bar{\alpha}$ of $f\bmod p^k$ is `associated' to a {\em unique} $\Z_p$-root $\alpha_i$ of $f$ : $\bar{\alpha}$ closely approximates $\alpha_i$ but is quite far from other $\alpha_j$s, for all $j\in [n], j\neq i$. So, root-set of $f\bmod p^k$ can be partitioned into $n$ subsets $S_{k,i}$, $i\in [n]$, where neighborhood $S_{k,i}$ is the set of those roots of $f\bmod p^k$ which are associated to $\Z_p$-root $\alpha_i$.

Say, multiplicity of root $\alpha_i$ is $e_i$, then $f(x)=: (x-\alpha_i)^{e_i} f_i(x)$ over $\Z_p$, where $f_i(\alpha_i)\ne0$. We call $f_i$ the {\em $\alpha_i$-free part} of $f$. Then, for $\bar{\alpha}$ to be a root of $f\bmod p^k$ we must have $f(\bar{\alpha})=(\bar{\alpha}-\alpha_i)^{e_i}\cdot f_i(\bar{\alpha})\equiv 0\bmod p^k$. Lemma \ref{lemma-uni-val} says that $f_i$ possesses equal valuation $\nu_i$, for all roots of $f\bmod p^k$ associated to $\alpha_i$, i.e, ones in $S_{k,i}$. That is, the maximum power of $p$ dividing $f_i(\bar{\alpha})$ is the same as that for $f_i(\bar{\beta})$, as long as $\bar{\alpha},\bar{\beta}\in S_{k,i}$. Note that, $v_p\left( (\bar{\alpha}-\alpha_i)^{e_i}\cdot f_i(\bar{\alpha}) \right) \ge k$ iff $v_p\left( (\bar{\alpha}-\alpha_i) \right) \ge (k-\nu_i)/e_i$. 

Eventually, these two lemmas together give us the size of the neighborhood, $|S_{k,i}|=p^{k-\ceil{(k-\nu_i)/e_i}}$. Moreover, the neighborhoods disjointly cover all the roots of $f\bmod p^k$. Hence, $N_k(f)=\sum_{i=1}^{n} |S_{k,i}|$. This is a formula for $N_k(f)$, when $k$ is large. But still the two parameters $\nu_i$ and $e_i$ are unknown as, unlike \cite{zuniga2003computing}, we are not provided the factorization of $f$ over $\Z_p$ (nor  could we find it in deterministic poly-time).

To compute $\nu_i, e_i$, we take help of root-counting algorithm of \cite{dwivedi2019counting}, which gives us the value of $N_k(f)$, and the underlying root-set structure  that it developed. We show that each representative-root $\bar{\alpha}_i$ of $f\bmod p^k$ is indeed the neighborhood $S_{k,i}$ (Theorem \ref{thm-property}) shedding new light on the root-set mod prime-powers.

Now we can get two equations, for the two unknowns $\nu_i, e_i$, by calling twice the algorithm of \cite{dwivedi2019counting}-- first for $k=k_i$ and second for $k=k_i+e_i$, where $k_i$ is such that $(k_i-\nu_i)/e_i$ is an integer (eg.~we can try all $k_i$ in the range $[k_0,\dots,k_0+\deg(f)]$). So, we can efficiently compute $\nu_i, e_i$ for a particular {\em representative-root} $\bar{\alpha}_i$, $i\in [n]$. So, this calculation also reveals some new parameters of representative-roots which were not mentioned in earlier related works \cite{berthomieu2013polynomial, dwivedi2019counting}.

\section{Preliminaries}

\subsection{Root-set of a univariate polynomial mod prime-powers}
\label{sec-root-set}


We recall a structural property (and related objects) of the root-set of univariate polynomials in the ring $\Z/\langle p^k\rangle$ \cite{dwivedi2019counting, dwivedi2019efficiently}. It says,

\begin{proposition}
\label{prop-root-set}
The root-set of an integral univariate polynomial $f$, over the ring of integers modulo prime powers, is the disjoint union of at most $\deg(f)$ many, efficiently representable, subsets.
\end{proposition}

We call these efficiently representable subsets-- {\em representative-roots}, as was defined and named in \cite[Sec.~2]{dwivedi2019efficiently}. This property of root-sets in $\Z/\langle p^k\rangle$ is indeed a generalization of the property of root-sets over a field: there are at most $\deg(f)$ many roots of $f(x)$ in a field.

To present representative-roots formally, we first reiterate some notations from \cite[Sec.~2]{dwivedi2019efficiently}.

\smallskip\noindent
\textbf{Representatives. } An abbreviation $*$ will be used to denote all of the underlying ring $R$. So for the ring $R= \Z/\langle p^k\rangle$, $*$ denotes all the $p^k$ distinct elements. Perceiving any element of $R$ in base-$p$ representation, like $x_0+p x_1+\ldots+p^{k-1}x_{k-1}$ where $x_i\in \{ 0,\ldots,p-1\}$ for all $i\in \{ 0,\ldots,k-1\}$, the set $\mb{a}:= a_0+p a_1+\ldots+ p^{l-1}a_{l-1} + p^l*$ `represents' the set of all the elements of $R$ which are congruent to $a_0+p a_1+\ldots+p^{l-1}a_{l-1} \bmod p^l$. Throughout the paper we call such sets {\em representatives} and we denote them using bold small letters, like $\mb{a}, \mb{b}$ etc. 

Let us denote the {\em length} of a representative $\mb{a}$ by $|\mb{a}|$, so if $\mb{a}:= a_0+p a_1+\ldots+ p^{l-1}a_{l-1} + p^l*$ then its length is $|\mb{a}|=l$. Now we formally define representative-roots of a univariate polynomial in $\Z/\langle p^k\rangle$.


\begin{definition}[Representative-roots]
\label{def-rep-root}
A set $\mb{a}=a_0+p a_1+\ldots+p^{l-1}a_{l-1}+p^l*$ is called a {\em representative-root} of $f(x)$ modulo $p^k$ if each $\alpha\in \mb{a}$ is a root of $f(x)\bmod p^k$, but, not all $\beta\in \mb{b}:= a_0+p a_1+\ldots+p^{l-2}a_{l-2}+p^{l-1}*$ are roots of $f(x)\bmod p^k$.
\end{definition}



It was first observed in the work of \cite{berthomieu2013polynomial} that there are at most $\deg(f)$-many representative-roots and they gave an efficient randomized algorithm to compute all these representative-roots (for a simple exposition of the algorithm see \cite[Sec.B]{dwivedi2019efficiently}).

We need a deterministic algorithm for our purpose (in Section  
\ref{sec-main-3}) to count, if not find, the representative-roots (as well as count the roots in each representative-root). So we use the deterministic poly-time algorithm of \cite{dwivedi2019counting} which returns all these representative-roots implicitly in the form of a data-structure they call--- {\em maximal split ideals} (MSI). The two explicit parameters, {\em length} and {\em degree} of a MSI immediately gives the count on the number of representative-roots (as well as roots) encoded by them, which suffices for our purpose. A similar idea to use triangular ideals for encoding roots was first appeared in the work of \cite{cheng2017counting}, to count roots deterministically but for `small' $k$.

We now define MSI from \cite[Sec.~2]{dwivedi2019counting}.

\begin{definition}{\cite[Sec.~2]{dwivedi2019counting} {\em(Maximal Split Ideals)}}
\label{def-MSI}
A triangular ideal $I=\langle h_0(x_0),\ldots,h_l(x_0,\ldots,x_l)\rangle$, where $0\leq l\leq k-1$ and each $h_i(x_0,\ldots,x_i)\in \F_p[x_0,\ldots,x_i]$, is called a {\em maximal split ideal} of $f(x)\bmod p^k$ if,

\begin{enumerate}

\item  the number of common zeros of $h_0,\ldots,h_l$ in ${\F_p}^{l+1}$ is $\prod_{i=0}^{l} \deg_{x_i}(h_i)$, where $\deg_{x_i}$ denotes the individual degree wrt $x_i$, and,

\item  for every common zero $(a_0,\ldots,a_l)\in {\F_p}^{l+1}$ of $h_0,\ldots,h_l$; $f(x)$ vanishes identically modulo $p^k$ with the substitution $x\to a_0+p a_1+\ldots+p^l a_l+p^{l+1}x$ but not with $x\to a_0+\ldots+p^{l-1} a_{l-1}+p^{l}x$.

\end{enumerate}

\end{definition} 

For an MSI $I$ given by its generators $h_0(x_0),\ldots,h_l(x_0,\ldots,x_l)$ we define its {\em length} to be $l+1$ and {\em degree}, denoted as $\deg(I)$, to be the number of common zeros of its generators which is $\prod_{i=0}^{l} \deg_{x_i}(h_i)$ by definition. 

Essentially, $I$ is encoding some representative-roots of $f\bmod p^k$ in the form of common roots of its generators. Indeed, Condition (2) of the definition is similar to that of representative-roots. If $(a_0,\ldots,a_l)$ is a common zero of the generators then by Condition (2), $a_0+p a_1+\ldots+p^l a_l+p^{l+1}*$ follows all the conditions to be a representative-root. It is apparent then :

\begin{lemma}{\cite[Lem.~6 \& 8]{dwivedi2019counting}}
\label{lemma-MSI-rep-root}
The length of an MSI $I$ is the length of each representative-root encoded by it and the degree of $I$ is the count on them. Thus, we get the count on the roots of $f\bmod p^k$ encoded by $I$ as ${\prod_{i=0}^{l} \deg_{x_i}(h_i)}\times p^{k-l-1}$.

\end{lemma}

We state the result of \cite{dwivedi2019counting} which returns all the representative-roots, in MSI form, in deterministic polynomial time.

\begin{theorem}{\em (Compute $N_k(f)$ \cite{dwivedi2019counting})}
\label{thm-dms}
In deterministic poly($|f|,k\log p$)-time one gets the maximal split ideals which collectively contain exactly the representative-roots of a univariate polynomial $f(x)\in \Z[x]$ modulo prime-power $p^k$. 

Moreover, using Lemma \ref{lemma-MSI-rep-root} we can count them, and all the roots of $f\bmod p^k$, in deterministic poly-time.
\end{theorem}

\subsection{Some definitions and notation related to $f$}
\label{sec-def}


We are given an integral univariate polynomial $f(x)\in \Z[x]$ of degree $d$ with coefficients magnitude at most $C\in \N$, and a prime $p$. Then, $f$ can also be thought of as an element of $\Z_p[x]$ (as $\Z\subseteq \Z_p$), where $\Z_p$ is the {\em ring of integers of $p$-adic rational numbers $\Q_p$}. In such a field $\Q_p$ (called non-archimedean {\em local} field) there exists a valuation function $v_p:\Q_p\to \Z\cup \{\infty\}$. Formally, the {\em valuation $v_p(a)$} of $a\in \Z_p$ ($\Z_p$ is a UFD) is defined to be the highest power of $p$  dividing $a$, when $a\neq 0$, and $\infty$ when $a=0$. This definition extends to the rationals $\Q_p$ naturally as $v_p(a/b):=v_p(a)-v_p(b)$, where $b\neq 0$ and $a,b\in \Z_p$ (see \cite{koblitz1977p} as reference text). 

Now we define the factors of $f$ in $\Z_p[x]$ as follows (note: we do not require $f$ to be monic).

\begin{definition}
\label{def-alphai-gj}
Let the $p$-adic integral factorization of $f$, into coprime irreducible factors, be 
$$f(x) \;=:\;  \prod_{i\in[n]}(x-\alpha_i)^{e_i} \cdot \prod_{j=1}^{m} g_j(x)^{t_j}$$
where each $\alpha_i$ is a $\Z_p$-root of $f$ with {\em multiplicity} $e_i$. Each $g_j(x)\in \Z_p[x]$ has multiplicity $t_j$; it is irreducible over $\Z_p$ and has no $\Z_p$-root. 
\end{definition}

For example, over $\Z_2$, $f = 2x^2+3x+1 = (x+1)\cdot (2x+1)$ has $n=m=1$.

\begin{definition}
\label{def-fi-vi}
For each $i\in [n]$, we define $f_i(x)\in \Z_p[x]$, called {\em $\alpha_i$-free part of $f$}, as $f_i(x):= f(x)/(x-\alpha_{i})^{e_i}$. We denote {\em valuation} $v_p(f_i(\alpha_i))$ as $\nu_i$, for all $i\in [n]$.
\end{definition}

\noindent
\textbf{Radical} of a univariate polynomial $h(x)$ over a field $\F$ is defined to be the univariate polynomial, denoted by $\rad(h)$, which is the product of coprime irreducible factors of $h$. This gives rise to following definition.

\begin{definition}
\label{def-rad}
Define $\rad(f):=( \prod_{i=1}^{n}(x-\alpha_i)) \cdot( \prod_{j=1}^{m} g_j(x))$. Analogously, the radical of $f_i$, for each $i\in [n]$, is defined as $\rad(f_i):= \rad(f)/(x-\alpha_i)$.
\end{definition}

\noindent
\textbf{Discriminant} of a polynomial $h(x)\in \F[x]$ is defined as $ D(h):= h_m^{2m-1}\cdot\prod_{1\le i<j\le m} (r_i-r_j)^2 $, where $\F$ is a field, $r_i$'s are the roots of $h(x)$ over the algebraic closure $\bar{\F}$, degree of $h$ is $m$, and $h_m$ is its leading coefficient.

\smallskip
The discriminant $D(h)$ is an element of $\F$. It is clear by the definition: all the roots of $h$ are distinct iff $D(h)\neq 0 $.
Eg.~discriminant of radical is nonzero.

\begin{definition}
\label{def-delta}
We denote by $\Delta$, the valuation with respect to $p$ of the discriminant of radical of $f$, i.e, $\Delta:= v_p(D(\rad(f)))$.
\end{definition}

We see that $\Delta$ must be finite, since roots of $\rad(f)$ are distinct. The following fact is easily established by the definition of discriminant and the fact that $\alpha_1,\ldots,\alpha_n$ are also roots of $\rad(f)$.

\begin{fact}
\label{fact-delta}
For $i\neq j \in [n]$, we have $v_p(\alpha_i-\alpha_j)\leq \Delta/2< \infty$.
\end{fact}

For our algorithm, $\Delta$ will be crucial in informing us about the behavior of the roots of $f\bmod p^k$.

\smallskip\noindent
\textbf{Properties of discriminant:}\label{disc-property}
\begin{enumerate}
\item  
Over $\Z_p$, if $u(x)| w(x)$ then $D(u)\mid D(w)$ and $v_p(D(u))\leq v_p(D(w))$.

\item
Discriminant of a linear polynomial is defined to be $1$.

\item
 If $w(x)=(x-a)\cdot u(x)$ then by the definition of discriminant, it is clear that $D(w)=D(u)\cdot u(a)^2$.

\item
 Discriminant $D(h)$ of a degree-$l$ univariate polynomial $h(x):= h_l x^l+\ldots+h_1 x+h_0$, over $\Z_p$, is also a multivariate polynomial over $\Z_p$ in the coefficients $h_0,\ldots,h_l$ (see \cite[Ch.1]{lidl1994introduction}). Moreover, it is computable in time polynomial in size of given $h$ (eg.~using determinant of a Sylvester matrix \cite[Ch.11, Sec.2]{von2013modern}).
\end{enumerate}


\section{Proof of main results}

\subsection{Interplay of $\Z_p$-roots and $\left(\Z/\langle p^k\rangle\right)$-roots}
\label{sec-main-1}

In this section we will establish a connection between $\left(\Z/\langle p^k\rangle\right)$-roots and $\Z_p$-roots of the given $f$, when $k$ is sufficiently large i.e, $k>d\Delta$ (see Sec.~\ref{sec-def} for the related notation).

Recall that $\alpha_1,\ldots,\alpha_n$ are the distinct $\Z_p$-roots of $f$ (Defn.~\ref{def-alphai-gj}). The following claim establishes a notion of `closeness' of any $\bar{\alpha}\in \Z_p$ to an $\alpha_j$. Later we will apply this to a representative-root $\bar{\alpha}$.

\begin{claim}[Close to a root]
\label{claim3}
For some $j\in [n], \bar{\alpha}\in \Z_p$, if $v_p(\bar{\alpha}-\alpha_j)>\Delta/2$,  then $v_p(\bar{\alpha}-\alpha_i)=$ $v_p(\alpha_j-\alpha_i)\leq \Delta/2$, for all $i\neq j, i\in [n]$.
\end{claim}
\begin{proof}
$v_p(\bar{\alpha}-\alpha_i)=$ $v_p(\bar{\alpha}-\alpha_j+\alpha_j-\alpha_i)$. Since $v_p(\bar{\alpha}-\alpha_j)>\Delta/2$ and $v_p(\alpha_j-\alpha_i)\leq \Delta/2$ (by Fact \ref{fact-delta}), we deduce, $v_p(\bar{\alpha}-\alpha_i)=$ min$\{  v_p(\bar{\alpha}-\alpha_j),v_p(\alpha_j-\alpha_i)\}=v_p(\alpha_j-\alpha_i)\leq \Delta/2$.
\end{proof}

The following lemma says that an irreducible can not take values with ever increasing valuation.

\begin{lemma}[Valuation of an irreducible]
\label{lemma-irred}
Let $h(x)\in \Z_p[x]$ be a polynomial with no $\Z_p$-root, and discriminant $D(h)\ne0$. Then, for any $\bar{\alpha}\in \Z_p$ , $v_p(h(\bar{\alpha}))\leq$ $v_p(D(h))$ . 
\end{lemma}
\begin{proof}
We give the proof by contradiction i.e, we show that if $v_p(h(\bar{\alpha}))> v_p(D(h))$, then $h(x)$ has a root in $\Z_p$.

Define $v_p(D(h))=: d(h)$. Let $\bar{\alpha}\in \Z_p$ such that $h(\bar{\alpha})\equiv 0 \bmod p^{\delta}$, for $\delta>d(h)$. Then we write, $h(x)=(x-\bar{\alpha})\cdot h_1(x) \;+\; p^{\delta}\cdot h_2(x)$. The two things to note here are:

(1). $D(h)\equiv D(h\bmod p^{\delta})\bmod p^{\delta}$ by discriminant's Property $(4)$ in Sec.\ref{disc-property}. Also, $D(h)\neq 0$ is given.

(2). Let $h'(x)$ be the first derivative of $h(x)$ and let $i:=v_p(h'(\bar{\alpha}))$. Then, we claim that $\delta>d(h)\geq 2i$. 

Consider $h'(x)= h_1(x)+(x-\bar{\alpha}){h}_1'(x)+p^{\delta}{h'}_2(x)$. So, $h'(\bar{\alpha})\equiv$ $h_1(\bar{\alpha})\bmod p^{\delta}$.
By Property (3) (Sec.~\ref{sec-def}) of discriminants, $D(h)\equiv D((x-\bar{\alpha})\cdot h_1(x)) \equiv D(h_1)\cdot h_1(\bar{\alpha})^2 \equiv D(h_1)\cdot h'(\bar{\alpha})^2 \bmod p^{\delta}$. Togetherwith $D(h)\ne0 \bmod p^{\delta}$, we deduce, $2i\leq d(h)< \delta$.

\smallskip
Now, we show that the root $\bar{\alpha}$ of $h\bmod p^{\delta}$ {\em lifts} to roots of $h$ mod $p^{\delta+j}$, for all $j\in \Z^{+}$. This is due to Hensel's Lemma (see \cite[Ch.~15]{von2013modern}); for completeness we give the proof.

By Taylor expansion, we have $h(\bar{\alpha}+p^{\delta-i}x)=$ $h(\bar{\alpha})+ h'(\bar{\alpha})\cdot p^{\delta-i}x+ h''(\bar{\alpha})\cdot p^{2(\delta-i)}x^2/2! +\cdots $. 

Note that there exists a unique solution $x_0\equiv (-h(\bar{\alpha})/h'(\bar{\alpha})p^{\delta-i}) \bmod p$ : $h(\bar{\alpha}+p^{\delta-i}x_0)\equiv 0\bmod p^{\delta+1}$. This follows from the Taylor expansion and since $2(\delta-i)> \delta$.

So, $\bar{\alpha}-p^{\delta-i}(h(\bar{\alpha})/h'(\bar{\alpha})p^{\delta-i})\bmod p^{\delta+1}$ is a lift, of $\bar{\alpha} \bmod p^{\delta}$. By a similar reasoning, it can be lifted further to arbitrarily high powers  $p^{\delta+j}$. Thus, proving that $h(x)$ has a $\Z_p$-root; which is a contradiction.
\end{proof}

The following lemma is perhaps the most important one. It associates every root $\bar{\alpha}$ of $f(x) \bmod p^k$ to a unique $\Z_p$-root of $f$. Recall the notation from Section \ref{sec-def}.

\begin{lemma}[Unique association]
\label{lemma-uni-assoc}
Let $k>d(\Delta+1)$ and $\bar{\alpha}\in \Z_p$ be a root of $f(x)\bmod p^k$. There exists a unique $\alpha_i$ such that $v_p(\bar{\alpha}-\alpha_i)>\Delta+1$ and thus, $v_p(\bar{\alpha}-\alpha_i)>v_p(\alpha_i-\alpha_j)$, for all $j\neq i, j\in [n]$.
\end{lemma}
\begin{proof}

Let us first prove that there exists some $i\in [n]$, given $\bar{\alpha}$, such that $v_p(\bar{\alpha}-\alpha_i)>\Delta+1$. For the sake of contradiction, assume that $v_p(\bar{\alpha}-\alpha_i)\leq \Delta+1$ for all $i\in [n]$. Then, by Definition \ref{def-alphai-gj},
 $v_p(f(\bar{\alpha}))=$ $\sum_{i=1}^{n} e_i\cdot v_p(\bar{\alpha}-\alpha_i)\; +$ $\sum_{j=1}^{m} t_j\cdot v_p(g_j(\bar{\alpha})) \le$
 $(\Delta+1)\cdot \sum_{i=1}^{n} e_i\; +$ $\sum_{j=1}^{m} t_j\cdot v_p(g_j(\bar{\alpha}))$ .

Since $g_j$ has no $\Z_p$-root, for all $j\in [m]$, by Lemma \ref{lemma-irred}, $v_p(g_j(\bar{\alpha}))\leq v_p(D(g_j))$. By the  properties given in Sec.~\ref{sec-def} we get: $v_p(D(g_j))\leq v_p(D(\rad(f)))=\Delta$; proving that $v_p(g_j(\bar{\alpha}))\leq \Delta$.

Going back, $v_p(f(\bar{\alpha}))\leq (\Delta+1)\cdot (\sum_{i=1}^{n}e_i + \sum_{j=1}^{m} t_j ) \leq$ $d(\Delta+1)< k$. It implies that $f(\bar{\alpha})\not\equiv 0\bmod p^k$; which contradicts the hypothesis that $\bar{\alpha}$ is a root of $f\bmod p^k$.

Thus, $\exists i\in [n]$, $v_p(\bar{\alpha}-\alpha_i)>\Delta+1$. The uniqueness of $i$ follows from Claim \ref{claim3}. 
\end{proof}

Having seen that every root $\bar{\alpha}$, of $f\bmod p^k$, is associated (or close) to a unique $\Z_p$-root $\alpha_i$, the following lemma tells us that the valuation of $\alpha_i$-free part of $f$ (resp.~factors of $f$ with no $\Z_p$-root) is the {\em same} on any $\bar{\alpha}$ close to $\alpha_i$. This unique valuation is important in getting an expression for $N_k(f)$.

\begin{lemma}[Unique valuation]
\label{lemma-uni-val}
Fix $i\in [n]$. Fix $\bar{\alpha}\in \Z_p$ such that $v_p(\bar{\alpha}-\alpha_i)>\Delta$. Recall $g_j(x), f_i$ from Section \ref{sec-def}. Then, 
\begin{enumerate}
\item $v_p(g_j(\bar{\alpha}))=v_p(g_j(\alpha_i))$, for all $j\in [m]$,

\item $v_p(f_i(\bar{\alpha}))=v_p(f_i(\alpha_i))$.
\end{enumerate}
\end{lemma}
In other words, valuation with respect to $p$ of $f_i = f(x)/(x-\alpha_i)^{e_i}$, on $x\mapsto \bar{\alpha}$, is fixed uniquely to $\nu_i:= v_p(f_i(\alpha_i))$,  for any `close' approximation $\bar{\alpha}\in \Z_p$ of $\alpha_i$.
\begin{proof}
 
Since $g_j \mid \rad(f_{i})$ and $\rad(f_{i}) \mid \rad(f)$, we have by the properties of discriminants (Sec.~\ref{sec-def}): $v_p( g_j(\alpha_i) ) \leq v_p( \rad(f_{i})(\alpha_i) ) \leq \Delta$, for all $j\in [m]$.

Since $v_p(\bar{\alpha}-\alpha_i)>\Delta$, we deduce $v_p( g_j(\bar{\alpha}) - g_j(\alpha_i) ) > \Delta$. Furthermore, $v_p( g_j(\alpha_i) ) \leq \Delta$ implies: $v_p(g_j(\bar{\alpha})) = v_p(g_j(\alpha_i))$. This proves the first part.

By Claim \ref{claim3}, $v_p(\bar{\alpha}-\alpha_u) = v_p(\alpha_i-\alpha_u)$, for all $u\neq i, u\in [n]$. Also, by the first part, $v_p(g_w(\bar{\alpha})) = v_p(g_w(\alpha_i))$, for all $w\in [m]$. Consequently, $v_p(f_i(\bar{\alpha})) = \sum_{u=1, u\neq i}^{n} e_u\cdot  v_p(\alpha_i-\alpha_u) \;+$ $\sum_{w=1}^{m} t_w\cdot   
 v_p(g_w(\alpha_i)) =$ $v_p(f_i(\alpha_i))$. This proves the second part.
\end{proof}

\subsection{Representative-roots versus neighborhoods}
We now connect the $\Z_p$-roots of $f$ to the representative-roots (defined in Section \ref{sec-root-set}) of $f\bmod p^k$.  Later we characterize each representative-root as a `neighborhood' in Theorem \ref{thm-property}.

\begin{lemma}[Perturb a root]
\label{lemma-lift}
Let $k>d(\Delta+1)$ and $\bar{\alpha}$ be a root of $f(x)\bmod p^k$ with $l:=v_p(\alpha_i-\bar{\alpha}) > \Delta+1$, for some $i\in [n]$ (as in Lemma \ref{lemma-uni-assoc}). Then, every $\bar{\beta}\in \bar{\alpha}+p^l *$ is also a root of $f(x)\bmod p^k$.

\end{lemma}
\begin{proof}
Since $f(\bar{\alpha})\equiv 0\bmod p^k$, we have $v_p(f(\bar{\alpha}))\geq k$. Using Lemma \ref{lemma-uni-val} we have $v_p(f_i(\bar{\alpha}))=v_p(f_i(\alpha_i))=\nu_i$. Thus, $v_p(f(\bar{\alpha}))= v_p(\alpha_i-\bar{\alpha})\cdot e_i+ v_p(f_i(\bar{\alpha})) = v_p(\alpha_i-\bar{\alpha})\cdot e_i+\nu_i$ $\ge k$. 
%

Similarly, $v_p(f(\bar{\beta}))=v_p(\alpha_i-\bar{\beta})\cdot e_i+v_p(f_i(\bar{\beta})) = v_p(\alpha_i-\bar{\beta})\cdot e_i+\nu_i \geq v_p(\alpha_i-\bar{\alpha})\cdot e_i+\nu_i$. 
Last inequality follows from  $v_p(\alpha_i-\bar{\beta})\geq l = v_p(\alpha_i-\bar{\alpha})$ .

From the above two paragraphs we get, $v_p(f(\bar{\beta}))\ge k$. Hence, $f(\bar{\beta})\equiv 0 \bmod p^k$.
\end{proof}

Now we define a notion of `neighborhood' of a $\Z_p$-root of $f$.

\begin{definition}[Neighborhood]
\label{def-S_ki}
For $i\in [n]$, $k>d(\Delta+1)$, we define {\em neighborhood} $S_{k,i}$ of $\alpha_i$ mod $p^k$ to be the set of all those roots of $f\bmod p^k$, which are close to the $\Z_p$-root $\alpha_i$ of $f$. Formally,
$$S_{k,i} \;:=\; \{ \bar{\alpha}\in \Z/\langle p^k\rangle \mid v_p(\bar{\alpha}-\alpha_i)>\Delta+1, f(\bar{\alpha})\equiv 0\bmod p^k\} \;.$$
\end{definition}

The notion of representative-root was first given in \cite{dwivedi2019efficiently}. Below we discover its new properties which will lead us to the understanding of {\em length} of a representative-root; which in turn will give us the size of a neighborhood contributing to $N_k(f)$.

\begin{theorem}[Rep.root is a neighborhood]
\label{thm-property}
Let $k>d(\Delta+1)$ and $\mb{a}:=a_0+p a_1+p^2 a_2+\ldots+p^{l-1} a_{l-1}+p^{l}*$ be a representative-root of $f(x)\bmod p^k$. Define the $\Z_p$-root reduction $\bar{\alpha}_i:= \alpha_i \bmod p^k$, for all $i\in [n]$. Fix an $i\in[n]$, then,

\begin{enumerate}
\item Length of $\mb{a}$ is {\em large}. Formally, $l>\Delta+1$.

\item If $\bar{\alpha}_i\in \mb{a}$, then $\bar{\alpha}_j \not\in \mb{a}$ for all $j\neq i, j\in [n]$. (This means with Lemma \ref{lemma-uni-assoc}: $\mb{a}$ has a {\em uniquely associated} $\Z_p$-root.)

\item If $\mathbf{a}$ contains $\bar{\alpha}_i$ then it also contains the respective {\em neighborhood}. In fact, if $\bar{\alpha}_i\in \mb{a}$,  then $S_{k,i} = \mb{a}$.

\end{enumerate}
\end{theorem}
\begin{proof}

\begin{enumerate}
\item  Consider $\bar{\alpha}:= a_0+p a_1+\ldots+p^{l-1} a_{l-1}$. By Lemma \ref{lemma-uni-assoc}, there is a unique $s\in[n]$: $v_p(\bar{\alpha}-\alpha_s)> \Delta+1$. Suppose $l\leq \Delta+1$. Then, 
$v_p(\bar{\alpha}+p^{\Delta+1} -\alpha_s)= \Delta+1$. As, $\bar{\alpha}':= (\bar{\alpha}+p^{\Delta+1})$ is also in $\mb{a}$, it again has to be close to a unique $\alpha_t$, $s\ne t\in[n]$ : $v_p(\bar{\alpha}'-\alpha_t)> \Delta+1$.
In other words, $\alpha_s+p^{\Delta+1} \equiv \bar{\alpha}+p^{\Delta+1} \equiv \alpha_t \bmod p^{\Delta+2}$. Thus, $v_p(\alpha_s-\alpha_t)= \Delta+1> \Delta/2$; contradicting Fact \ref{fact-delta}. This proves $l> \Delta+1$.

\item Consider distinct $\bar{\alpha}_i, \bar{\alpha}_j \in\mb{a}$. Then, by the definition of $\mb{a}$, we have
$v_p(\bar{\alpha}_i-\bar{\alpha}_j) \ge l> \Delta+1> \Delta/2$; contradicting Fact \ref{fact-delta}. Thus, there is a unique $i$.

\item  Suppose there exists a neighborhood element $\bar{\beta}\not\in \mb{a}$, satisfying the conditions  $v_p(\alpha_i-\bar{\beta})>\Delta+1$ and $f(\bar{\beta})\equiv 0\bmod p^k$. 
Let $j$ be the index of the first coordinate where $\bar{\beta}$ and $\mb{a}$ differ; so, $j<l$ since $\bar{\beta}\not\in \mb{a}$.
Clearly, $j>\Delta+1$; otherwise, since $\bar{\alpha}_i\in \mb{a}$ and $\bar{\beta}\not\in \mb{a}$, we deduce $v_p(\alpha_i-\bar{\beta})= j \leq \Delta+1$; which is a contradiction. 

By $v_p(\alpha_i-\bar{\beta})=j>\Delta+1$ and Lemma \ref{lemma-lift}, we get: every element in $\bar{\beta}+p^{j}*$ is a root of $f(x)\bmod p^k$. Consequently, each element in  $a_0+p a_1+p^2 a_2+\ldots+p^{j-1} a_{j-1}+p^{j}*$
 is a root of $f(x)\bmod p^k$; which contradicts that $\mb{a}$ is a representative-root ($\because j<l$, see Def.~\ref{def-rep-root}). 
Thus, $\bar{\beta}\in \mb{a}$; implying $S_{k,i}\subseteq \mb{a}$.

Conversely, consider $\bar{\alpha}\in \mb{a}$. Then, as before, $v_p(\bar{\alpha}_i-\bar{\alpha}) \ge l> \Delta+1$; implying that $\bar{\alpha}\in S_{k,i}$. Thus, $S_{k,i}\supseteq \mb{a}$..
\end{enumerate}
\end{proof}


Next, we get the expression for the length of a representative-root.

\begin{theorem}
\label{thm-root}

For $k>d(\Delta+1)$, the representative-roots, of $f(x)\bmod p^k$, are in a one-to-one correspondence with $\Z_p$-roots of $f$. Moreover, the length of the representative-root $\mb{a}$, corresponding to $\alpha_i$, is 
$l_{i,k} := \ceil{(k-\nu_i)/e_i}$.

\end{theorem}

\begin{proof}

By Proposition \ref{prop-root-set}, every root of $f\bmod p^k$ is in exactly one of the representative-roots. So each reduced $\Z_p$-root $\bar{\alpha}_i:=\alpha_i\bmod p^k$ is in a unique representative-root. Thus, by Theorem \ref{thm-property} parts (2) \& (3), we get the one-to-one correspondence as claimed.

Consider a $p$-adic integer $\bar{\alpha}$ with $v_p(\bar{\alpha}-\alpha_i)=: l_{\bar{\alpha}}> \Delta$. We have the following equivalences: $\bar{\alpha}\in\mb{a}$ iff 
$v_p(f(\bar{\alpha}))\geq k$ iff
$v_p((\bar{\alpha}-\alpha_i)^{e_i}\cdot f_i(\bar{\alpha}))\geq k$ iff
$e_i l_{\bar{\alpha}}+\nu_i\geq k$ (by Lemma \ref{lemma-uni-val}) iff
$l_{\bar{\alpha}}\geq \ceil{(k-\nu_i)/e_i} =l_{i,k}$.

Write the representative-root corresponding to $\alpha_i$ as $\mb{a}=:a_0+p a_1+p^2 a_2+\ldots+p^{l-1} a_{l-1}+p^{l}*$. Clearly, $l \,=\, \min \{ l_{\bar{\alpha}} \mid \bar{\alpha}\in \mb{a} \} \,\ge\, l_{i,k}$. Note that if $l>l_{i,k}$ then by the equivalences we could reduce the length $l$ of the representative-root $\mb{a}$; which is a contradiction. Thus, $l=l_{i,k}$.
\end{proof}

\subsection{Formula for $N_k(f)$-- Proof of Corollary \ref{cor1} }
\label{sec-main-2}

For large enough $k$, the earlier section gives us an easy way to count the roots. In fact, we have the following simple formula for $N_k(f)$.

\begin{theorem}[Roots mod $p^k$]
\label{thm-formula}

For $k>d(\Delta+1)$, $N_{k}(f)=\sum_{i\in[n]} p^{ k - \ceil{(k-\nu_i)/e_i} }$, where clearly $\nu_i, e_i$ and $n$ (as in Sec.\ref{sec-def}) are independent of $k$.

\end{theorem}
\begin{proof}

Fix $i\in[n]$ and $k>d(\Delta+1)$. By Theorem \ref{thm-root} we get that in the unique representative-root $\mathbf{a}$, corresponding to $\alpha_i \bmod p^k$, the $\left( k - \ceil{(k-\nu_i)/e_i}\right)$-many higher-precision coordinates could be set arbitrarily from $[0\ldots p-1]$ (while the rest lower-precision ones are fixed).  That gives us the count via contribution for each $i\in[n]$. 
Moreover, the sum over neighborhoods, for each $i\in[n]$, gives us exactly $N_k(f)$. 

Also, note that if $n=0$ then the count $N_k(f)=0$.
\end{proof}


\begin{proof}[Proof of Corollary \ref{cor1}]

Theorem \ref{thm-formula} gives a closed form expression for $N_{k}(f)$, when $k\geq k_0:= d(\Delta+1)+1 \le d(2d-1)(\log_p C+\log_p d)+1$.

For the other part, note that the discriminant $D(f)\ne0$  iff $f$ is squarefree. In the squarefree case $e_i=1$, for all $i\in [n]$. By Theorem \ref{thm-formula}, $N_{k}(f)=\sum_{i\in[n]} p^{\nu_i}$; which is independent of $k$.
\end{proof}

\subsection{Computing Poincar\'e series-- Proof of Theorem \ref{thm-main} }
\label{sec-main-3}

Building up on the ideas of the previous sections, we will show how to deterministically compute Poincar\'e series $P(t)=\sum_{k=0}^{\infty} N_{k}(f)(p^{-1}t)^{k}$, associated to the input $f(x)$, efficiently; thereby proving Theorem \ref{thm-main}. Before that, we need few notation as follows.

\smallskip\noindent 
Set $k_0:= d(\Delta+1)+1$ so we know by Theorem \ref{thm-formula} that for $k\geq k_0$, $N_k(f)=\sum_{i=1}^{n} N_{k,i}(f)$, where $N_{k,i}(f):=p^{k-\ceil{(k-\nu_i)/e_i}}$. For each $i\in [n]$, define $k_i$ to be the least integer such that $k_i\geq k_0$ and $(k_i-\nu_i)/e_i$ is an integer. Then, Poincar\'e series $P(t)$ can be partitioned into finite and {\em infinite} sums as, 
$$ P(t) \;=\; P_{0}(t) \,+\, \sum_{i=1}^{n} P_{i}(t)$$ 
\noindent
where $P_{i}(t):=\sum_{k=k_i}^{\infty} N_{k,i}(f)\cdot(p^{-1}t)^{k}$ and $P_0(t):= \left( \sum_{k=0}^{k_0 -1} N_{k}(f)\cdot (p^{-1}t)^{k} \right) \;+\; \sum_{i=1}^{n}\sum_{k=k_0}^{k_i -1} N_{k,i}(f)\cdot(p^{-1}t)^{k}$.

We now compute the multiplicity $e_i$ by viewing it as the {\em step} that increments the length, of the representative-root associated to $\alpha_i$, as $k$ keeps growing above $k_0$.

\begin{lemma}[Compute $e_i$]
\label{lemma-vi-ei}
We can compute the number of $\Z_p$-roots $n$ of $f$ as well as $k_i, \nu_i$ and $e_i$, for each $i\in [n]$, in deterministic poly($d, \log C+\log p$) time.
\end{lemma}
\begin{proof}

By Theorem \ref{thm-dms}, we get all the representative-roots of $f\bmod p^k$ implicitly in the form of maximal split ideals (in short we  call split ideals). By Lemma \ref{lemma-MSI-rep-root}, the length of a split ideal is also the length of all the representative-roots represented by it and the degree is the number of representative roots represented by it.  Since by Theorem \ref{thm-root}, $n$ is also the number of representative roots of $f\bmod p^k$ for $k\geq k_0$ hence we run the algorithm of Theorem \ref{thm-dms} for $k=k_0$ and sum up the degree of all the split ideals obtained, to get $n$. 

Suppose the split ideal $I$ we find contains a representative-root $\mb{a}$ of $f\bmod p^k$ corresponding to $\alpha_i$, with $k_i$ as defined before. How do we compute $k_i$? By Theorem \ref{thm-root} length of $\mb{a}$, when $k=k_i$, is $l_{i,k_i}=(k_i-\nu_i)/e_i$. Now, for all $k=k_i+1, k_i+2,\ldots,k_i+e_i$, the length $l_{i,k}$ remains equal to $l_{i,k_i}+1$; while for the next $k=k_i+e_i+1$, $l_{i,k}$ increments by one.
 
So we run the algorithm of Theorem \ref{thm-dms}, for several $k\ge k_0$. When we find the length incrementing by one, namely, at the two values $k=k_i+1$ and $k=k_i':=k_i+1+e_i$, then we have found $e_i$ (and $k_i$). 
From the equation, $k_i-\nu_i= e_i\cdot l_{i,k_i}$ , we also find $\nu_i$.

Suppose the split ideal $I$ we find contains {\em two} representative-roots $\mb{a}$ and $\mb{b}$ mod $p^k$, corresponding to $\Z_p$-roots $\alpha_i$ and $\alpha_j$ respectively, such that $e_i\neq e_j$ (wlog say $e_i<e_j$). In this case, even if $\mb{a}$ and $\mb{b}$ have the same length, when $k=k_i$, they will evolve to different length representative-roots when we go to a `higher-precision' arithmetic mod $p^{k_i+1+e_i}$ (by formula in Thm.~\ref{thm-root}). So $\mb{a}, \mb{b}$ must lie in different length split ideals, say $I_a$ and $I_b$ respectively.

Now, for another representative-root $\mb{c}$ in $I_a$, say corresponding to $\alpha_s$, we have $e_i=e_s$ and hence $\nu_i=\nu_s$. By computing $e_i$ and $\nu_i$ as before, now using the length of $I$ and $I_a$, we compute $e_s$ and $\nu_s$ (and $k_s$) for every $\mb{c}$ in $I_a$. Since, by Lemma \ref{lemma-MSI-rep-root}, the degree of $I_a$ is the number of such representative-roots in $I_a$, we could compute $n$; moreover, we get $k_i, \nu_i, e_i$ for all $i\in [n]$.

Clearly, we need to run the algorithm of Theorem \ref{thm-dms} at most $2\max_{i\in [n]}\{ e_i\}= O(d)$ times, to study the evolution of split ideals (implicitly, that of the underlying representative-roots). Also $\Delta$ is logarithm (to base $p$) of the determinant of a Sylvester matrix which gives $\Delta=O(d\cdot(\log_p C+\log_p d))$. So, the algorithm is poly-time as claimed.
\end{proof}

Now we prove that the infinite sums $P_i(t)$ are formally equal to rational functions of $t=p^{-s}$.

\begin{lemma}[Infinite sums are rational]
\label{lemma-converge}

For each $i\in [n]$, the series $P_i(t)$ is a rational function of $t$ as, 
\[ P_i(t) \;=\; \frac{t^{k_i}\cdot(p-t(p-1)-t^{e_i})}{p^{(k_i-\nu_i)/e_i}\cdot(1-t)\cdot(p-t^{e_i})}. \]

\end{lemma}
\begin{proof}

Recall that $P_i(t)=\sum\limits_{k=k_i}^{\infty} N_{k,i}(f)\cdot (p^{-1}t)^{k}$. For simplicity write $T:= p^{-1}t$ and define integer $\delta_i:= k_i-(k_i -\nu_i)/e_i$. Now $P_i$ can be rewritten using residues mod $e_i$ as, 
$$ P_i(t) \;=\; \sum\limits_{l=k_i}^{k_i+e_i-1}\sum\limits_{k=0}^{\infty} N_{l+ke_i,i}(f)\cdot T^{l+ke_i} \;. $$ 
For simplicity take $l=k_i$ and consider the sum, $\sum\limits_{k=0}^{\infty} N_{k_i+ke_i,i}(f)\cdot T^{k_i+ke_i}$. 
We find that $N_{k_i,i}(f)= p^{\delta_i}, N_{k_i+e_i,i}(f)= p^{\delta_i+e_i-1}, N_{k_i+2e_i,i}(f)= p^{\delta_i+2(e_i-1)}$, and so on. Hence,

$\sum\limits_{k=0}^{\infty} N_{k_i+ke_i,i}(f)\cdot T^{k_i+ke_i} \;=\; p^{\delta_i}T^{k_i}\cdot[1+p^{e_i-1}T^{e_i}+(p^{e_i-1}T^{e_i})^2+\ldots]$
$ =\; p^{\delta_i}\cdot T^{k_i}/ \left( 1-p^{e_i-1}T^{e_i} \right)$.
\begin{align*}
\text{So,  } P_i(t) &\;=\; \frac{p^{\delta_i}T^{k_i}}{1-p^{e_i-1}T^{e_i}}+\frac{p^{\delta_i}T^{k_i+1}}{1-p^{e_i-1}T^{e_i}}+\frac{p^{\delta_i+1}T^{k_i+2}}{1-p^{e_i-1}T^{e_i}}+\ldots+\frac{p^{\delta_i+e_i-2}T^{k_i+e_i-1}}{1-p^{e_i-1}T^{e_i}} \\
&\;=\; \frac{p^{\delta_i}T^{k_i}}{1-p^{e_i-1}T^{e_i}}+ \frac{p^{\delta_i}T^{k_i+1}}{1-p^{e_i-1}T^{e_i}}\cdot \left( 1+pT+(pT)^2+\ldots+(pT)^{e_i-2} \right) \\
&\;=\; \frac{p^{\delta_i}T^{k_i}}{1-p^{e_i-1}T^{e_i}}\cdot \left( 1+T\cdot \frac{1-(pT)^{e_i-1}}{1-pT} \right) \;.
\end{align*}

Putting $T=t/p$ and $\delta_i= k_i-(k_i -\nu_i)/e_i$ we get,

\[ P_i(t)= \frac{t^{k_i}(p-t(p-1)-t^{e_i})}{p^{(k_i -\nu_i)/e_i}(1-t)(p-t^{e_i})}. \]

\end{proof}

Now we are in a position to prove our main theorem.

\begin{proof}[Proof of Theorem \ref{thm-main}]

Recall $P(t)= P_{0}(t)+\sum_{i=1}^{n} P_{i}(t)$. We first compute $P_0(t)$ which is the sum of two polynomials in $t$, namely, $Q_1(t):=\sum_{j=0}^{k_0 -1} N_{j}(f)(p^{-1}t)^{j}$ of degree $O(d\Delta)$, and $Q_2(t)=\sum_{i=1}^{n}\sum_{l=k_0}^{k_i -1} N_{l,i}(f)(p^{-1}t)^{l}$ also of degree $O(d\Delta)$. By a standard determinant/Sylvester matrix calculation one shows: $d\Delta\le O\left( d^2\cdot (\log_p C+\log_p d) \right) $.

We can compute the polynomial $Q_1(t)$ in deterministic poly($d,\log C+\log p$)-time by calling the root-counting algorithm of \cite{dwivedi2019counting} (Theorem \ref{thm-dms}) $k_0-1$ times, getting each $N_j(f)$, for $j=1,\ldots,k_0 -1$ (note: $N_0(f):=1$).

Polynomial $Q_2(t)$ is a sum of $n\leq d$ polynomials,  each with $k_i-k_0\leq d$ many simple terms. Using Lemma \ref{lemma-vi-ei}, we can compute each $\nu_i, e_i$, hence, $N_{l,i}(f)$. So, computation of $Q_2$ again takes time poly($d, \log C+\log p$).

Lemma \ref{lemma-converge} gives us the rational form expression for $P_i(t)$, for each $i\in [n]$. So, using Lemma \ref{lemma-vi-ei} we can compute the Poincar\'e series \[ P(t)=P_0(t)+\sum\limits_{i=1}^{n} \frac{t^{k_i}(p-t(p-1)-t^{e_i})}{p^{(k_i -\nu_i)/e_i}(1-t)(p-t^{e_i})}\]
  in deterministic poly($d,\log C+\log p$) time.

By inspecting the above expression, the degree of denominator $B(t)$ is $1+\sum_{i=1}^{n} e_i = O(d)$. 
The degree of numerator $A(t)$ is $\le k_0+2d\le O\left( d^2\cdot (\log_p C+\log_p d) \right)$.
\end{proof}

\vspace{-1mm}
\section{Conclusion}
\vspace{-1mm}

We presented the first complete solution to the problem of computing Igusa's local zeta function for any given integral univariate polynomial and a prime $p$. Indeed, our methods work for given $f\in \Z_p[x]$ as our proof for integral $f$ goes by considering its factorization over $Z_p$ (Sec. \ref{sec-def}).

The next natural question to study is whether we could generalize our method to $n$-variate integral polynomials (say, $n=2$?). Note that for growing $n$ this problem is atleast NP-hard.

We also found explicit closed-form expression for $N_k(f)$ and efficiently computed the explicit parameters involved therein, which could be used to compute Greenberg's constants associated with a univariate $f$ and $p$. Greenberg's constants appear in a classical theorem of Greenberg \cite[Thm.~1]{greenberg1966rational} which is a generalization of Hensel's lemma to several $n$-variate polynomials.
We hope that our methods for one-variable case could be generalized to compute Greenberg's constants for $n$-variable case to give an effective version of Greenberg's theorem.

We also hope that our methods extend computing Igusa's local zeta function from characteristic zero ($\Z_p$) to positive characteristic ($\F_p[[T]]$) at least if some standard restrictions are imposed on the characteristic for eg. $p$ is `large-enough'. This is supported by the fact that the root counting algorithm of \cite{dwivedi2019counting} also extends to $\F[[T]]$ for a field $\F$.

\smallskip
\noindent
{\bf Acknowledgements. } We thank anonymous reviewers for their helpful comments and pointing out relevant references to improve the draft of the paper. In particular we thank them for their suggestion which greatly improved the conclusion section and for pointing out a connection to Greenberg's work. N.S.~thanks the funding support from DST (DST/SJF/MSA-01/2013-14) and N.~Rama Rao Chair.

\bibliographystyle{amsplain}

\bibliography{bibliography}



\end{document}